\documentclass{amsart}

\usepackage{amsmath, amssymb, amsthm}

\newtheorem{theorem}{Theorem}[section]
\newtheorem{lemma}[theorem]{Lemma}
\newtheorem{proposition}[theorem]{Proposition}

\newtheorem{claim}[theorem]{Claim}

\newtheorem{question}[theorem]{Question}

\newtheorem*{THMA}{Theorem A}
\newtheorem*{THMB}{Theorem B}

\theoremstyle{definition}
\newtheorem{definition}[theorem]{Definition}
\newtheorem{remark}[theorem]{Remark}

\newcommand{\cf}{\mathrm{cf}}
\newcommand{\dom}{\mathrm{dom}}
\newcommand{\bb}{\mathbb}

\newcommand{\otp}{\mathrm{otp}}

\newcommand{\tp}{\mathrm{tp}}
\newcommand{\ssup}{\mathrm{ssup}}

\title{On the growth rate of chromatic numbers of finite subgraphs}
\author{Chris Lambie-Hanson}
\thanks{We thank P\'{e}ter K\'{o}mjath for providing us with some information
about the history of the question answered in this paper.}
\address{Department of Mathematics and Applied Mathematics \\
Virginia Commonwealth University \\
Richmond, VA 23284 \\ United States}
\email{cblambiehanso@vcu.edu}

\subjclass[2010]{Primary 05C63. Secondary 05C15, 03E05.}
\keywords{Uncountable graphs, chromatic number, club guessing}

\begin{document}
\begin{abstract}
	We prove that, for every function $f:\bb{N} \rightarrow \bb{N}$, there is a
	graph $G$ with uncountable chromatic number such that, for every $k \in \bb{N}$
	with $k \geq 3$, every subgraph of $G$ with fewer than $f(k)$ vertices has chromatic number
	less than $k$. This answers a question of Erd\H{o}s, Hajnal, and Szemeredi.
\end{abstract}
\date{\today}
\maketitle

\section{Introduction}

The De Bruijn-Erd\H{o}s compactness theorem, proven in \cite{de_bruijn_erdos},
states that, for every natural number $d$ and every graph $G$, if every finite subgraph of $G$ has chromatic
number at most $d$, then $G$ also has chromatic number at most $d$. It follows
that, for every graph $G$ with infinite chromatic number, we can define a function
$f_G : \bb{N} \rightarrow \bb{N}$ by letting $f_G(k)$ be the least natural number
$m$ for which there exists a subgraph of $G$ with $m$ vertices and chromatic number at least $k$.
This function is clearly increasing, and the question naturally arises:
how quickly can $f_G$ grow?

By a result of Erd\H{o}s \cite{erdos_circuits_and_subgraphs},
for any function $f:\bb{N} \rightarrow \bb{N}$,
there is a graph $G$ with chromatic number $\aleph_0$ such that $f_G$ grows
faster than $f$. For this reason, investigation of this question has been
focused on graphs with uncountable chromatic number. Indeed, there are relevant
ways in which graphs with uncountable chromatic number behave fundamentally differently
from graphs with finite or countable chromatic number. For example, for every
natural number $k$, there are graphs of girth $k$ and arbitrarily large finite
chromatic number. By taking disjoint unions of such graphs, it follows that there
are graphs of girth $k$ and chromatic number $\aleph_0$. On the other hand,
Erd\H{o}s and Hajnal prove in \cite{erdos_hajnal_chromatic_number} that every
uncountably chromatic graph must contain a copy of every finite bipartite graph;
in particular, it contains a cycle of every even length.

In \cite{erdos_hajnal_chromatic_number}, Erd\H{o}s and Hajnal introduce the
shift graphs $G_0(\alpha, n, s)$ for ordinals $\alpha$, and natural numbers
$n$ and $s$ with $1 \leq s \leq n-1$. The following facts about shift graphs are proven in
\cite{erdos_hajnal_chromatic_number} and \cite{erdos_hajnal_szemeredi}.
In what follows, the expression $\log^{(n)}$ denotes the $n$-times iterated base 2 logarithm
and $\exp_n$ denotes the $n$-times iterated exponential function.

\begin{theorem}
	Suppose that $\alpha$ is an ordinal and $n$ and $s$ are natural numbers with
	$1 \leq s \leq n-1$.
	\begin{enumerate}
		\item If $\kappa$ is an infinite cardinal and $\alpha \geq (\exp_{n-1}(\kappa))^+$,
		then $\chi(G_0(\alpha, n, s)) > \kappa$.
		\item There is a constant $c_n > 0$ such that, for every natural number $k$
		and every subgraph $H$ of $G_0(\alpha, n, 1)$ with $k$ vertices, we have
		$\chi(H) \leq c_n \log^{(n-1)}(k)$.
	\end{enumerate}
\end{theorem}
As a result, it follows that, for every $n < \omega$, there is an uncountably
chromatic graph $G$ such that $f_G$ grows more quickly than $\exp_n$.
This led Erd\H{o}s, Hajnal, and Szemeredi to formulate the following general question,
asking if $f_G$ can grow arbitrarily quickly for graphs $G$ with uncountable
chromatic number.
(The question is mentioned in many places in slightly different forms; see, e.g.,
\cite{erdos_hajnal_szemeredi} and \cite{erdos_unsolved_problems}).

\begin{question} \label{main_question}
	Is it true that, for every function $f:\bb{N} \rightarrow \bb{N}$, there is
	an uncountably chromatic graph $G$ such that $\lim_{k \rightarrow \infty} f(k)/f_G(k) = 0$?
\end{question}

In \cite{komjath_shelah_finite_subgraphs}, Komj\'{a}th and Shelah prove that Question~\ref{main_question}
consistently has a positive answer. In particular, given a model of $\mathsf{ZFC}$,
they construct a forcing extension of that model in which, for every function $f:\bb{N}
\rightarrow \bb{N}$, there is a graph $G$ such that $|G| = \chi(G) = \aleph_1$ and,
for every natural number $k \geq 3$, $f_G(k) \geq f(k)$.

In this paper, we prove outright that Question~\ref{main_question} has a positive
answer in $\mathsf{ZFC}$.

\begin{THMA}
	For every function $f:\bb{N} \rightarrow \bb{N}$, there is a graph
	$G$ such that $|G| = 2^{\aleph_1}$, $\chi(G) = \aleph_1$ and, for every natural
	number $k \geq 3$, $f_G(k) \geq f(k)$.
\end{THMA}

Notice that the cardinality of the graph given by the theorem is strictly greater than
$\aleph_1$. It is unclear whether this is necessary in general. However, we also
prove that, under the additional assumption of $\diamondsuit$, we can obtain
graphs of size $\aleph_1$ and can even require them to be particular types of
graphs known as \emph{Hajnal-M\'{a}t\'{e} graphs}.

\begin{THMB}
	Suppose that $\diamondsuit$ holds. Then, for every function
	$f:\bb{N} \rightarrow \bb{N}$, there is a Hajnal-M\'{a}t\'{e} graph
	$G$ such that $|G| = \chi(G) = \aleph_1$ and, for every natural
	number $k \geq 3$, $f_G(k) \geq f(k)$.
\end{THMB}

The structure of the paper is as follows. In Section~\ref{graph_section}, we
introduce some of the basic graph-theoretic notions we will be using and prove
some basic propositions. In Section~\ref{club_guessing_section}, we review
the set-theoretic technology of club guessing and relate it to some of the
notions introduced in Section~\ref{graph_section}. In Section~\ref{main_proof_section},
we prove Theorem A. In Section \ref{hajnal_mate_section}, we prove Theorem B.
We conclude by noting some questions that remain open.

\subsection{Notation and conventions}

The notation and definitions used here are mostly standard.
We refer the reader to \cite{jech} for any undefined set-theoretic notions.
We include $0$ in the set $\bb{N}$ of natural numbers.
$\mathrm{Ord}$ denotes the class of ordinals. Any ordinal is thought of as a set
whose elements are all strictly smaller ordinals. If $n$ is a natural
number, then $[\mathrm{Ord}]^n$ denotes the class of $n$-element sets of
ordinals. Members of $[\mathrm{Ord}]^n$ will typically be presented
as $\{\alpha_0, \alpha_1, \ldots, \alpha_{n-1}\}$, where $\alpha_0 < \alpha_1 <
\ldots < \alpha_{n-1}$.

If $A$ is a set of ordinals, then $\otp(A)$ denotes the order type of $A$.
If $i < \otp(A)$, then $A(i)$ denotes the unique
element $\alpha$ of $A$ such that $\otp(A \cap \alpha) = i$. If
$I \subseteq \otp(A)$, then $A[I]$ denotes $\{A(i) \mid i \in I\}$.
The \emph{strong supremum of $A$}, denoted $\ssup(A)$, is defined
to be $\sup\{\alpha + 1 \mid \alpha \in A\}$. It is the least ordinal
that is strictly greater than every ordinal in $A$.
If $A$ and $B$ are two sets of ordinals, then we say that
$B$ \emph{end-extends} $A$, written $A \sqsubseteq B$, if
$B \cap \ssup(A) = A$. If $\sigma$ and $\tau$ are functions whose
domains are sets of ordinals, then we say that $\tau$ \emph{end-extends}
$\sigma$, written $\sigma \sqsubseteq \tau$, if $\dom(\sigma)
\sqsubseteq \dom(\tau)$ and $\tau \restriction \dom(\sigma) = \sigma$.

The \emph{cofinality} of an ordinal $\alpha$ is denoted by $\cf(\alpha)$.
If $\beta$ is an ordinal and $\mu$ is an infinite regular cardinal, then
$S^\beta_\mu = \{\alpha < \beta \mid \cf(\alpha) = \mu\}$.
If $\theta$ is an infinite cardinal, then $H(\theta)$ denotes the collection of
sets hereditarily of cardinality less than $\theta$.

All of our graphs are simple undirected graphs (in particular, they have no loops).
If $G = (V, E)$ is a graph and $\vec{v} = \langle v_0, \ldots, v_n \rangle$ is a finite sequence
of elements of $V$, then $\vec{v}$ is a \emph{walk} in $G$ if $\{v_i, v_{i+1}\} \in E$
for every $i < n$. $\vec{v}$ is a \emph{closed walk} if $\vec{v}$ is a walk and, moreover,
$v_0 = v_n$. Finally, $\vec{v}$ is a \emph{cycle} if it is a closed walk,
$n \geq 3$, and $v_i \neq v_j$ for all $i < j < n$. In such a case, $n$ is the
\emph{length} of the cycle. The cycle is an \emph{odd cycle} if $n$ is odd.

If $G = (V, E)$ is a graph, then we will sometimes write $|G|$ to mean $|V|$.
The chromatic number of $G$ is denoted by $\chi(G)$.

\section{Types and Specker graphs} \label{graph_section}

In this section, we introduce some of the basic notions that we will be using.

\begin{definition}
	Suppose that $n$ is a natural number.
	\begin{enumerate}
		\item A \emph{disjoint type of length $n$} is a function $t : 2n \rightarrow 2$
		such that
		\[
		|\{i < 2n \mid t(i) = 0\}| = |\{i < 2n \mid t(i) = 1\}| = n.
		\]
		\item If $a$ and $b$ are disjoint elements of $[\mathrm{Ord}]^n$, then
		$\tp(a,b)$ is the unique disjoint type $t :2n \rightarrow 2$ such that,
		letting $a \cup b = \{\alpha_0, \alpha_1, \ldots, \alpha_{2n-1}\}$,
		enumerated in increasing order, we have $a = \{\alpha_i \mid i < 2n \text{ and } t(n) = 0\}$
		and $b = \{\alpha_i \mid i < 2n \text{ and } t(n) = 1\}$.
	\end{enumerate}
\end{definition}

A disjoint type of length $n$ will often be represented as a sequence of $0$s and
$1$s of length $2n$. For example, if $a = \{0, 1, 3\}$ and $b = \{2, 4, 5\}$,
then $\tp(a, b) = 001011$. If $t_0$ and $t_1$ are two disjoint types of lengths
$n_0$ and $n_1$, respectively, then $t_0 {^\frown} t_1$ denotes the disjoint type
of length $n_0 + n_1$ represented by the concatenation of the sequences of $0$s
and $1$s representing $t_0$ and $t_1$. Formally, $t_0 {^\frown} t_1$ is the function
$t : 2n_0 + 2n_1 \rightarrow 2$ defined by letting
\[
	t(i) =
	\begin{cases}
		t_0(i) & \text{if } i < 2n_0 \\
		t_1(i - 2n_0) & \text{if } 2n_0 \leq i < 2n_0 + 2n_1.
	\end{cases}
\]
We will be particularly interested in the following
family of types.

\begin{definition}
	Suppose that $s$ and $n$ are natural numbers with $1 \leq s \leq n-1$. Then
	$t^n_s$ is the disjoint type of length $n$ defined by letting, for all
	$i < 2n$,
	\[
		t^n_s(i) =
		\begin{cases}
			0 & \text{if } i < s \\
			0 & \text{if } s \leq i < 2n - s \text{ and } i - s \text{ is even} \\
			1 & \text{if } s \leq i < 2n - s \text{ and } i - s \text{ is odd} \\
			1 & \text{if } i \geq 2n - s.
		\end{cases}
	\]
\end{definition}

This definition might initially be difficult to parse. Essentially, $t^n_s$ is
the type consisting of $s$ copies of $0$, followed by $n-s$ copies of $01$, followed
by $s$ copies of $1$. For example, $t^5_2 = 0001010111$.

\begin{definition}
	Suppose that $n$ is a natural number, $t$ is a disjoint type of length $n$, and
	$\alpha$ is an ordinal. Then $G(\alpha, t) = ([\alpha]^n, E(\alpha, t))$ is the
	graph with vertex set $[\alpha]^n$ and edge set $E(\alpha, t)$ defined by setting
	$\{a,b\} \in E(\alpha, t)$ if and only if $a$ and $b$ are disjoint elements of
	$[\alpha]^n$ and either $\tp(a, b) = t$ or $\tp(b, a) = t$.
\end{definition}

Graphs of the form $G(\alpha, t^n_s)$, where $s$ and $n$ are natural numbers
with $1 \leq s \leq n-1$, are sometimes known as \emph{Specker graphs} (see
\cite{erdos_hajnal_szemeredi}). Erd\H{o}s and Hajnal proved the following
facts about Specker graphs.

\begin{theorem}[{\cite[Theorem 7.4]{erdos_hajnal_chromatic_number}}] \label{specker_graph_theorem}
	Suppose that $s$ and $n$ are natural numbers with $1 \leq s \leq n-1$, and
	suppose that $\alpha$ is an ordinal.
	\begin{enumerate}
		\item If $\alpha$ is an infinite cardinal, then $\chi(G(\alpha, t^n_s)) =
		|G(\alpha, t^n_s)| = \alpha$.
		\item If $n \geq 2s^2 + 1$, then $G(\alpha, t^n_s)$ contains no odd
		cycles of length $2s + 1$ or shorter.
	\end{enumerate}
\end{theorem}

\begin{remark}
	The functions $f_G$ for $G$ of the form $G(\alpha, t^n_1)$ were investigated
	in \cite{avart_luczak_rodl}.
\end{remark}

We end this section with some basic propositions about graphs that will be
useful for us in the proof of Theorem A.

\begin{proposition} \label{chromatic_product}
	Suppose that $G = (V,E)$ is a graph, $k \in \bb{N}$, and $E = \bigcup_{j < k} E_j$.
	For $j < k$, let $G_j = (V, E_j)$. Then $\chi(G) \leq \prod_{j < k} \chi(G_j)$.
\end{proposition}

\begin{proof}
	For each $j < k$, let $c_j:V \rightarrow \chi(G_j)$ be a proper coloring of
	$G_j$. Let $\mathcal{H}$ be the set of functions $h:k \rightarrow \mathrm{Ord}$ such that,
	for all $j < k$, $h(k) < \chi(G_j)$. Clearly, $\mathcal{H} = \prod_{j < k} \chi(G_j)$.
	Define a coloring $c:V \rightarrow \mathcal{H}$
	by letting $c(v)(j) = c_j(v)$ for all $v \in V$ and $j < k$. It is easily
	verified that $c$ is a proper coloring of $G$.
\end{proof}

\begin{proposition} \label{walks_to_cycles_prop}
	Suppose that $j \geq 1$ is a natural number, $G$ is a graph, and $G$
	has a closed walk of length $2j + 1$. Then $G$ has an odd cycle of length
	$2j + 1$ or shorter.
\end{proposition}

\begin{proof}
	The proof is by induction on $j$. To take care of the case $j = 1$, simply
	note that a closed walk of length $3$ must be a cycle (since our graphs
	have no loops). Suppose now that $j > 1$ and $\vec{v} = \{v_0, v_1, \ldots, v_{2j + 1}\}$
	is a closed walk in $G$. If $\vec{v}$ is a cycle, then we are done. Otherwise,
	by rotating the walk if necessary, we may assume that there is $i$ with $0 < i < 2j$ such
	that $v_0 = v_i$. Then $\{v_0, \ldots, v_i\}$ and $\{v_i, \ldots, v_{2j + 1}\}$
	are both closed walks in $G$. One of them must have an odd length
	(and neither can have a length of $1$, since our graph has no loops), so
	we can appeal to the inductive hypothesis to obtain our desired conclusion.
\end{proof}

\begin{proposition} \label{graph_homomorphism_prop}
	Suppose that $m \in \bb{N}$, $G = (V_G, E_G)$ and $H = (V_H, E_H)$ are graphs,
	$H$ has no odd cycles of length $m$ or shorter, and there is a graph
	homomorphism from $G$ to $H$. Then $G$ has no odd cycles of length $m$ or shorter.
\end{proposition}

\begin{proof}
	Let $\varphi:V_G \rightarrow V_H$ induce a graph homomorphism from $G$ to $H$.
	If $j \geq 1$ and $\{v_0, \ldots, v_{2j + 1}\}$ is an odd cycle in $G$, then $\{\varphi(v_0), \ldots,
	\varphi(v_{2j + 1})\}$ is a closed path in $H$. By Proposition \ref{walks_to_cycles_prop},
	$H$ must then contain an odd cycle of length $2j + 1$ or shorter. The result
	follows.
\end{proof}

\section{Club guessing} \label{club_guessing_section}

In this section, we review the machinery of club guessing, which will be used in
the proof of Theorem A, and then prove a key lemma regarding the interaction
between club guessing and disjoint types.

\begin{definition}
	Suppose that $\kappa < \lambda$ are regular cardinals and $S \subseteq S^\lambda_\kappa$
	is stationary. A \emph{club-guessing sequence} on $S$ is a sequence
	$\vec{C} = \langle C_\alpha \mid \alpha \in S \rangle$ such that
	\begin{itemize}
		\item for every $\alpha \in S$, $C_\alpha$ is a club in $\alpha$ of order type $\kappa$;
		\item for every club $D$ in $\lambda$, there are stationarily many
		$\alpha \in S$ such that $C_\alpha \subseteq D$.
	\end{itemize}
\end{definition}

Shelah proved that, if there is at least a one-cardinal gap between $\kappa$
and $\lambda$, then club-guessing sequences always exist.

\begin{theorem}[Shelah \cite{cardinal_arithmetic}] \label{club_guessing_theorem}
	Suppose that $\kappa < \lambda$ are regular cardinals, $\kappa^+ < \lambda$,
	and $S \subseteq S^\lambda_\kappa$ is stationary. Then there is a club-guessing
	sequence on $S$.
\end{theorem}

\begin{proposition} \label{club_guessing_decomposition_prop}
	Suppose that $\kappa < \lambda$ are regular cardinals,
	$S \subseteq S^\lambda_\kappa$ is stationary, and $\langle C_\alpha \mid
	\alpha \in S \rangle$ is a club-guessing sequence. Suppose moreover that
	$\mu < \lambda$ and $S = \bigcup_{\eta < \mu} S_\eta$. Then there is
	$\eta < \mu$ such that $\langle C_\alpha \mid \alpha \in S_\eta \rangle$
	is a club-guessing sequence.
\end{proposition}

\begin{proof}
	Suppose not. This means that, for every $\eta < \mu$, there is a club
	$D_\eta \subseteq \lambda$ such that, for all $\alpha \in S_\eta$,
	$C_\alpha \not\subseteq D_\eta$. Let $D = \bigcap_{\eta < \mu} D_\eta$.
	Since $\lambda$ is regular and $\mu < \lambda$, it follows that $D$ is a club in $\lambda$.
	Also, $D \subseteq D_\eta$ for every $\eta < \mu$, so it follows that, for
	every $\eta < \mu$ and every $\alpha \in S_\eta$, $C_\alpha \not\subseteq
	D$. But $S = \bigcup_{\eta < \mu} S_\eta$, so, for all $\alpha \in S$,
	$C_\alpha \not\subseteq D$, contradicting the fact that
	$\langle C_\alpha \mid \alpha \in S \rangle$ is a club-guessing sequence.
\end{proof}

We now prove that the initial segments of the elements of a club-guessing sequence
indexed by a subset of $S^\lambda_\omega$ realize every disjoint type.

\begin{lemma} \label{type_lemma}
	Suppose that $\lambda$ is an uncountable regular cardinal,
	$S \subseteq S^\lambda_\omega$ is stationary, and $\vec{C} =
	\langle C_\delta \mid \delta \in S \rangle$ is a club-guessing sequence.
	Suppose moreover that $n < \omega$ and $t:2n \rightarrow 2$ is a
	disjoint type of length $n$. Then there are $\gamma < \delta$, both in $S$, such that
	$C_\gamma[n]$ and $C_\delta[n]$ are disjoint and $\tp(C_\gamma[n], C_\delta[n]) = t$.
\end{lemma}

\begin{proof}
	Let $\theta$ be a sufficiently large regular cardinal, let $\vartriangleleft$
	be a fixed well-ordering of $H(\theta)$, and let $\langle M_\xi \mid \xi <
	\lambda \rangle$ be a continuous $\in$-increasing
	chain of elementary submodels of $(H(\theta), \in, \vartriangleleft)$
	such that
	\begin{itemize}
		\item $\vec{C}, S \in M_0$;
		\item for all $\xi < \lambda$, $|M_\xi| < \lambda$;
		\item for all $\xi < \lambda$, $\beta_\xi := M_\xi \cap \lambda$
			is an ordinal.
	\end{itemize}
	Let $D = \{\beta_\xi \mid \xi < \lambda\}$.
	Then $D$ is a club in $\lambda$, so we can fix $\delta \in S$ such
	that $\beta_\delta = \delta$ and
	$C_\delta \subseteq D$. For each $i < \omega$, let $\xi_i$ be the unique ordinal $\xi$ such that
	$\beta_{\xi} = C_\delta(i)$, and let $N_i = M_{\xi_i}$. Let $N_\omega = M_\delta$. Then
	$\langle N_i \mid i \leq \omega \rangle$ is an $\in$-increasing chain of
	elementary submodels,
	$\delta = N_\omega \cap \lambda$ and, for $i < \omega$, $C_\delta(i) = N_i \cap \lambda$.

	Let $\exists^\infty \alpha$ stand for the quantifier ``there are unboundedly many $\alpha < \lambda$ such that \ldots".
	(This is the same as $\forall \eta < \lambda \exists \alpha < \lambda (\eta < \alpha \text{ and} \ldots)$.)
	Notice that, for all $\eta < C_\delta(n-1)$, there is $\beta > \eta$ for which there is $\gamma \in S$ such that
	$C_\gamma[n] = C_\gamma[n-1] \cup \{\beta\}$ (namely, $\beta = C_\delta(n-1)$ and $\gamma = \delta$ witness this statement). By
	elementarity, observing that $C_\delta[n-1] \in N_{n-1}$ it follows that
	\[
		N_{n-1} \models \exists^\infty \beta_{n-1} ~ \exists \gamma \in S \left(
		C_\gamma[n] = C_\delta[n-1] \cup \{\beta_{n-1}\} \right).
	\]
	By another application of elementarity, $H(\theta)$ satisfies the same statement, so, for all
	$\eta < C_\delta(n-2)$, there is $\beta > \eta$ for which the following statement holds:
	\[
		\exists^\infty \beta_{n-1} ~ \exists \gamma \in S \left(
		C_\gamma[n] = C_\delta[n-2] \cup \{\beta, \beta_{n-1}\}\right).
	\]
	Namely, $\beta = C_\delta(n-2)$ witnesses this statement. As above, we obtain
	\[
		N_{n-2} \models \exists^\infty \beta_{n-2} ~ \exists^\infty \beta_{n-1} ~ \exists \gamma \in S
		\left(C_\gamma[n] = C_\delta[n-2] \cup \{\beta_{n-2}, \beta_{n-1}\}\right).
	\]
	Again, it follows that $H(\theta)$ satisfies the same statement. Continuing in this way,
	the following statement holds in $H(\theta)$ and hence in every $N_i$:
	\[
		\exists^\infty \beta_0 \ldots \exists^\infty \beta_{n-1} ~ \exists \gamma \in S
		\left(C_\gamma[n] = \{\beta_0, \ldots, \beta_{n-1}\}\right).
	\]

	Define an auxiliary function $s:n \rightarrow n + 1$ as follows. Given
	$i < n$, let $s(i)$ be the number of $1$s appearing before the $i^{\mathrm{th}}$
	(starting at zero)
	$0$ in the sequence representation of $t$. For example, if $t = 001011$, then
	$s(0) = s(1) = 0$ and $s(2) = 1$.
	We now recursively choose $\beta_0^* < \beta_1^* < \ldots < \beta_{n-1}^*$, ensuring that, for
	all $i < n$, $\beta_i^* < C_\delta(s(i))$ and
	\[
		\exists^\infty \beta_{i+1} \ldots \exists^\infty \beta_{n-1} ~ \exists \gamma \in S
		\left(C_\gamma[n] = \{\beta_0^*, \ldots, \beta_i^*\} \cup \{\beta_{i+1}, \ldots, \beta_{n-1}\}\right).
	\]
	We will also arrange so that, if $s(i) > 0$, then $\beta_i^* > C_\delta(s(i) - 1)$.

	The construction is straightforward. If $i < n$ and we have already chosen $\{\beta_j^* \mid j < i\}$,
	then, by our recursion hypotheses, we have $\{\beta_j^* \mid j < i\} \in N_{s(i)}$ and
	\[
		N_{s(i)} \models \exists^\infty \beta_i \ldots \exists^\infty \beta_{n-1} ~ \exists \gamma \in S
		\left(C_\gamma[n] = \{\beta_j^* \mid j < i\} \cup \{\beta_i, \ldots, \beta_{n-1}\}\right),
	\]
	so we can choose $\beta_i^*$ witnessing this statement such that $\max\{\beta_j^* \mid j < i\} < \beta_i^*
	< C_\gamma(s(i))$ and such that, if $s(i) > 0$, then $\beta_i^* > C_\gamma(s(i) - 1)$.

	At the end of the construction, we have
	\[
		N_\omega \models \exists \gamma \in S \left(C_\gamma[n] = \{\beta_0^*, \ldots, \beta_{n-1}^*\}\right),
	\]
	so we can choose $\gamma \in S \cap \delta$ witnessing this statement. We constructed
	$\{\beta_i^* \mid i < n\}$ precisely so that $\tp(\{\beta_i^* \mid i < n\}, C_\delta[n]) = t$,
	so $\gamma$ and $\delta$ are as desired.
\end{proof}

\section{Proof of Theorem A} \label{main_proof_section}

We are now ready to prove Theorem A; we restate it here for convenience.

\begin{THMA}
	For every function $f:\bb{N} \rightarrow \bb{N}$, there is a graph
	$G$ such that $|G| = 2^{\aleph_1}$, $\chi(G) = \aleph_1$ and, for every natural
	number $k \geq 3$, $f_G(k) \geq f(k)$.
\end{THMA}

\begin{proof}
	Fix a function $f:\bb{N} \rightarrow \bb{N}$. We will construct a graph $G$
	such that $|G| = 2^{\aleph_1}$, $\chi(G) = \aleph_1$ and, for every $k \in
	\mathbb{N}$, if $H$ is a subgraph of $G$ with at most $f(k)$ vertices, then
	$\chi(H) \leq 2^{k+1}$. This clearly suffices for the theorem.

	For each $k \in \bb{N}$, let
	$s_k$ be the smallest natural number $s \geq 1$ such that $2s+1 \geq f(k)$, and let
	$n_k = 2s_k^2 + 1$. Note that, by Clause (2) of Theorem~\ref{specker_graph_theorem},
	the graph $G(\omega_2, t^{n_k}_{s_k})$ has no odd cycles of length $f(k)$
	or shorter. Partition $\bb{N}$ into adjacent intervals $\{I_0, I_1, \ldots \}$,
	with $|I_k| = n_k$. More precisely, set $I_0 = \{0, 1, \ldots, n_0 - 1\}$
	and, if $k \in \bb{N}$ and $I_k$ has been specified, then let $m_{k + 1} =
	\max(I_k) + 1$ and set $I_{k+1} = \{m_{k + 1}, m_{k + 1}+1, \ldots, m_{k+1} + n_{k+1}-1\}$.

	Let $S = S^{\omega_2}_{\omega}$. By Theorem~\ref{club_guessing_theorem},
	we can fix a club-guessing sequence $\vec{C} = \langle C_\alpha \mid \alpha
	\in S \rangle$. For $\beta \in S$, let $\Sigma_\beta$ be the set of
	all functions $\sigma:S \cap \beta \rightarrow \bb{N}$, let
	$\Sigma_{< \beta} = \bigcup_{\alpha \in S \cap \beta} \Sigma_\alpha$,
	and let $\Sigma = \bigcup_{\alpha \in S} \Sigma_\alpha$. For $\sigma \in
	\Sigma$, let $\alpha_\sigma$ be the unique $\alpha \in S$ such that
	$\sigma \in \Sigma_\alpha$. Clearly, $|\Sigma| = \aleph_2 \cdot
	\aleph_0^{\aleph_1} = 2^{\aleph_1}$. We will define a graph $G = (\Sigma, E)$
	that will be as desired.
	We will simultaneously be defining an auxiliary function
	$h:E \rightarrow \bb{N}$. These objects will satisfy the following
	requirements.
	\begin{enumerate}
		\item If $\sigma, \tau \in \Sigma$ and $\{\sigma, \tau\} \in E$, then
			either $\sigma \sqsubseteq \tau$ or $\tau \sqsubseteq \sigma$.
		\item For all $\sigma \sqsubseteq \tau$ in $\Sigma$ and all $k \in \bb{N}$, if
			$\{\sigma, \tau\} \in E$ and $h(\{\sigma, \tau\}) = k$, then
			\begin{enumerate}
				\item $\tau(\alpha_\sigma) = k$;
				\item for all $j \leq k$, we have that $C_{\alpha_\sigma}[I_j]$ and
					$C_{\alpha_\tau}[I_j]$ are disjoint and $\tp(C_{\alpha_\sigma}[I_j],
					C_{\alpha_\tau}[I_j]) = t^{n_j}_{s_j}$.
			\end{enumerate}
		\item For all $\tau \in \Sigma$ and all $k \in \bb{N}$, there is at most
			one $\sigma \sqsubseteq \tau$ such that $\{\sigma, \tau\} \in E$ and
			$h(\{\sigma, \tau\}) = k$.
	\end{enumerate}

	To define $G$, it clearly suffices to specify, for each $\tau \in \Sigma$,
	the set of $\sigma \sqsubseteq \tau$ such that $\{\sigma, \tau\} \in E$.
	The value of $h(\{\sigma, \tau\})$ for such $\sigma$ will then be determined by requirement
	$(2)(a)$ above as $h(\{\sigma, \tau\}) = \tau(\alpha_\sigma)$. To this end,
	fix $\beta \in S$ and $\tau \in \Sigma_\beta$. For each $k \in \bb{N}$, ask
	the following question: is there $\alpha \in S \cap \beta$ such that
	$\tau(\alpha) = k$ and, for all $j \leq k$, $C_\alpha[I_j]$ and
	$C_\beta[I_j]$ are disjoint and $\tp(C_\alpha[I_j], C_\beta[I_j]) =
	t^{n_j}_{s_j}$? If the answer is ``yes", then let $\alpha^\tau_k$
	be the least such $\alpha$ and place $\{\tau \restriction (S \cap
	\alpha^\tau_k), \tau\}$ in $E$. If the answer is ``no", then there will be no
	$\sigma \sqsubseteq \tau$ such that $\{\sigma, \tau\} \in E$ and $h(\sigma,
	\tau) = k$. This completes the construction of $G$; it is clear that we have
	satisfied the requirements listed above.

	For each $k \in \bb{N}$, set $E_k = \{\{\sigma, \tau\} \in E \mid h(\{\sigma,
	\tau\}) = k\}$ and $E_{\geq k} = \{\{\sigma, \tau\} \in E \mid h(\{\sigma,
	\tau\}) \geq k\}$.

	\begin{claim} \label{decomposition_claim}
		For all $k \in \bb{N}$,
		\begin{enumerate}
			\item $(\Sigma, E_k)$ is cycle-free;
			\item $(\Sigma, E_{\geq k})$ has no odd cycles of length $f(k)$ or shorter.
		\end{enumerate}
	\end{claim}

	\begin{proof}
		(1) Suppose that $k, \ell \in \bb{N}$ and $\langle \sigma_0, \ldots,
		\sigma_{\ell}\rangle$ enumerates a cycle in $(\Sigma, E_k)$. By rotating the
		elements if necessary, we may assume that $\alpha_{\sigma_0} \geq \alpha_{\sigma_j}$
		for all $j \leq \ell$. In particular, we have $\sigma_1, \sigma_{\ell - 1}
		\sqsubseteq \sigma_0$ and $h(\{\sigma_1, \sigma_0\}) = k = h(\{\sigma_{\ell - 1},
		\sigma_0\})$. But this contradicts requirement (3) in the construction of
		$G$.

		(2) Fix $k \in \bb{N}$. By requirement (2)(b) in our construction of $G$, the function $\varphi:\Sigma
		\rightarrow [\omega_2]^{n_k}$ defined by $\varphi(\sigma) = C_{\alpha_\sigma}
		[I_k]$ induces a graph homomorphism from $(\Sigma, E_{\geq k})$ to
		$G(\omega_2, t^{n_k}_{s_k})$. Since $G(\omega_2, t^{n_k}_{s_k})$ contains no
		odd cycles of length $f(k)$ or shorter, Proposition \ref{graph_homomorphism_prop}
		implies that $(\Sigma, E_{\geq k})$ also contains no such odd cycles.
	\end{proof}

	We can now show that the finite subgraphs of $G$ behave as desired.

	\begin{claim}
		Suppose that $k \in \bb{N}$ and $H = (V_H, E_H)$ is a subgraph of $G$
		with $|V_H| \leq f(k)$. Then $\chi(H) \leq 2^{k+1}$.
	\end{claim}

	\begin{proof}
		Note that $E_H = (E_H \cap E_{\geq k}) \cup \bigcup_{j < k}(E_H \cap
		E_j)$. By Clause (1) of Claim~\ref{decomposition_claim}, for all $j < k$,
		$H_j := (V_H, E_H \cap E_j)$ is cycle-free and hence has chromatic number at most $2$.
		By Clause (2) of Claim~\ref{decomposition_claim} and the fact that
		$|V_H| \leq f(k)$, it follows that $H_{\geq k} := (V_H, E_H \cap E_{\geq k})$ has no
		odd cycles and thus also has chromatic number at most 2. Proposition
		\ref{chromatic_product} then implies that
		\[
			\chi(H) \leq \chi(H_{\geq k}) \cdot \prod_{j < k}
			\chi(H_j) \leq 2 \cdot 2^k = 2^{k+1}.
		\]
	\end{proof}

	To finish the proof of the theorem, it remains to show that $\chi(G) =
	\aleph_1$. First note that, by construction, for each $\tau \in \Sigma$,
	there are only countably many $\sigma \sqsubseteq \tau$ with $\{\sigma,
	\tau\} \in E$. It is thus straightforward to define a proper coloring
	$c:\Sigma \rightarrow \omega_1$ of $G$ by recursion on $\alpha_\sigma$.
	Therefore, $\chi(G) \leq \aleph_1$.

	To see that $\chi(G) \geq \aleph_1$,
	suppose for sake of contradiction that $c:\Sigma \rightarrow \bb{N}$ is
	a proper coloring of $G$. Recursively define a function
	$\rho:S \rightarrow \bb{N}$ by letting
	\[
		\rho(\alpha) = c(\rho \restriction (S \cap \alpha))
	\]
	for all $\alpha \in S$. For $k \in \mathbb{N}$, let $S_k = \{\alpha \in S \mid \rho(\alpha) = k\}$.
	By Proposition~\ref{club_guessing_decomposition_prop}, we can fix
	$k \in \mathbb{N}$ such that $\{C_\alpha \mid \alpha \in S_k\}$ is a
	club-guessing sequence. By Lemma~\ref{type_lemma}, we can find
	$\alpha^* < \beta^*$ in $S_k$ such that, for all $j \leq k$,
	$C_{\alpha^*}[I_j]$ and $C_{\beta^*}[I_j]$ are disjoint and
	$\tp(C_{\alpha^*}[I_j], C_{\beta^*}[I_j]) = t^{n_j}_{s_j}$.
	(The disjoint type to which Lemma~\ref{type_lemma}
	is applied here is the concatenation $t^{n_0}_{s_0} {^\frown}
	t^{n_1}_{s_1} {^\frown} \ldots {^\frown} t^{n_k}_{s_k}$.)

	It follows that, when considering $\rho \restriction (S \cap \beta^*)$ in the
	construction of $G$, the answer to the question about $k$ was ``yes",
	since $\alpha^*$ is a witness. There is therefore an $\alpha \in S \cap
	\beta^*$ such that $\rho(\alpha) = k$ and $\{\rho \restriction (S \cap \alpha),
	\rho \restriction (S \cap \beta^*)\} \in E$. But, by our definition of $\rho$,
	we have
	\[
		c(\rho \restriction (S \cap \alpha)) = \rho(\alpha) = k = \rho(\beta^*)
		= c(\rho \restriction (S \cap \beta^*)),
	\]
	contradicting the assumption that $c$ is a proper coloring of $G$. Thus,
	$\chi(G) = \aleph_1$, so we have completed the proof.
\end{proof}

\section{Diamond and Hajnal-M\'{a}t\'{e} graphs} \label{hajnal_mate_section}

We begin this section by recalling the combinatorial principle $\diamondsuit$.

\begin{definition}
	$\diamondsuit$ is the assertion that there is a sequence $\langle A_\alpha
	\mid \alpha < \omega_1 \rangle$ such that
	\begin{enumerate}
		\item for every $\alpha < \omega_1$, $A_\alpha \subseteq \alpha$;
		\item for every $A \subseteq \omega_1$, there are stationarily many
		$\alpha < \omega_1$ for which $A \cap \alpha = A_\alpha$.
	\end{enumerate}
\end{definition}

$\diamondsuit$ is easily seen to be a strengthening of the Continuum Hypothesis,
and it holds in many canonical inner models of set theory, such as G\"{o}del's
constructible universe $\mathrm{L}$.

Recalling the previous section, notice that, if we had a club-guessing sequence on $S^{\omega_1}_\omega$,
then we could modify the proof of Theorem A to obtain graphs of
cardinality $2^{\aleph_0}$ witnessing its
conclusion (the vertex set of the graphs would be $\bigcup_{\alpha
\in S^{\omega_1}_\omega} \Delta_\alpha$, where $\Delta_\alpha$ is the set of all
functions from $S^\alpha_\omega$ to $\bb{N}$). Since $\diamondsuit$ implies both
the existence of such a club-guessing sequence and the Continuum Hypothesis,
$\diamondsuit$ implies the existence of such graphs of cardinality $\aleph_1$.
We can do slightly better than this, though, and ensure that these graphs
have a particular structure.

\begin{definition}
	A graph $G = (\omega_1, E)$ is a \emph{Hajnal-M\'{a}t\'{e} graph} if, for
	all $\beta < \omega_1$, the set $N_G^{<}(\beta) := \{\alpha < \beta \mid
	\{\alpha, \beta\} \in E\}$ is either finite or a set of order type $\omega$
	converging to $\beta$.
\end{definition}

In \cite{hajnal_mate}, Hajnal and M\'{a}t\'{e} prove that $\diamondsuit^+$,
which is a strengthening of $\diamondsuit$ that also holds in $\mathrm{L}$, implies
the existence of Hajnal-M\'{a}t\'{e} graphs with uncountable chromatic number.
On the other hand, they prove in the same paper that Martin's Axiom implies
that every Hajnal-M\'{a}t\'{e} graph has countable chromatic number. Komjath,
in \cite{komjath_note_on_hajnal_mate} proves that $\diamondsuit$ is sufficient to obtain uncountably chromatic
Hajnal-M\'{a}t\'{e} graphs that are, moreover, triangle-free.
In \cite{komjath_shelah_forcing_constructions}, Komjath and Shelah improve this
result and show that, for every natural number $k$, there is an uncountably
chromatic Hajnal-M\'{a}t\'{e} graph with no odd cycles of length $2k+1$ or
shorter.

We now show how to use $\diamondsuit$ to adjust the proof of Theorem A to
obtain Hajnal-M\'{a}t\'{e} graphs. For convenience, we restate Theorem B here.

\begin{THMB}
	Suppose that $\diamondsuit$ holds. Then, for every function
	$f:\bb{N} \rightarrow \bb{N}$, there is a Hajnal-M\'{a}t\'{e} graph
	$G$ such that $|G| = \chi(G) = \aleph_1$ and, for every natural
	number $k \geq 3$, $f_G(k) \geq f(k)$.
\end{THMB}

\begin{proof}
	Fix a function $f:\bb{N} \rightarrow \bb{N}$. We will construct a Hajnal-M\'{a}t\'{e}
	graph $G$ such that $\chi(G) = \aleph_1$ and, for every $k \in \bb{N}$, if
	$H$ is a subgraph of $G$ with at most $f(k)$ vertices, then $\chi(H) \leq
	2^{k+1}$. As in the proof of Theorem A, this will suffice to prove the theorem.
	For notational convenience, the vertex set of our graph will actually be
	$S:= S^{\omega_1}_\omega$ rather than $\omega_1$. Our graph can easily be transferred
	to a Hajnal-M\'{a}t\'{e} graph on $\omega_1$ by, for instance, using the
	unique order preserving map from $S^{\omega_1}_\omega$ to $\omega_1$.

	For $k \in \mathbb{N}$, define the natural numbers $s_k$ and $n_k$ and the interval
	$I_k$ exactly as in the proof of Theorem A. By a straightforward coding argument,
	$\diamondsuit$ is easily seen to be equivalent to the existence of a sequence
	$\langle (C_\alpha, f_\alpha) \mid \alpha \in S \rangle$ such that
	\begin{itemize}
		\item for all $\alpha \in S$, $C_\alpha$ is a cofinal subset of $\alpha$
		of order type $\omega$ and $f_\alpha:S \cap \alpha \rightarrow \bb{N}$;
		\item for every club $D \subseteq \omega_1$ and every function $f: S
		\rightarrow \bb{N}$, there are stationarily many $\alpha \in S$ such that
		\begin{itemize}
			\item $C_\alpha \subseteq D$;
			\item $f \restriction (S \cap \alpha) = f_\alpha$.
		\end{itemize}
	\end{itemize}
	Fix such a sequence. As in the proof of Theorem A, we will define our
	graph $G = (S, E)$ together with an auxiliary function $h:E \rightarrow \bb{N}$.
	These will satisfy the following requirements.
	\begin{enumerate}
		\item For all $\alpha < \beta$ in $S$ and all $k \in \bb{N}$, if $\{\alpha,
		\beta\} \in E$ and $h(\{\alpha, \beta\}) = k$, then
		\begin{enumerate}
			\item $f_\beta(\alpha) = k$;
			\item for all $j \leq k$, $C_\alpha[I_j]$ and $C_\beta[I_j]$
			are disjoint and $\tp(C_\alpha[I_j], C_\beta[I_j]) = t^{n_j}_{s_j}$.
		\end{enumerate}
		\item For all $\beta \in S$ and all $k \in \bb{N}$, there is at most one
		$\alpha < \beta$ such that $\{\alpha, \beta\} \in E$ and $h(\{\alpha, \beta\})
		= k$.
	\end{enumerate}
	To define $G$, it suffices to specify $N^{<}_G(\beta) := \{\alpha < \beta \mid
	\{\alpha, \beta\} \in E\}$ for each $\beta \in S$. The value of $h(\{\alpha, \beta\})$
	for $\alpha \in N^{<}_G(\beta)$ will then be forced to be $f_\beta(\alpha)$.
	To this end, fix $\beta \in S$. For each $k \in \mathbb{N}$, ask the following question:
	is there $\alpha \in S \cap \beta$ such that $f_\beta(\alpha) = k$ and, for all
	$j \leq k$, $C_\alpha[I_j]$ and $C_\beta[I_j]$ are disjoint and $\tp(C_\alpha[I_j],
	C_\beta[I_j]) = t^{n_j}_{s_j}$? If the answer is ``yes", then let $\alpha^\beta_k$
	be the least such $\alpha$ and place $\{\alpha^\beta_k, \beta\}$ in $E$. If the
	answer is ``no", then there will be no $\alpha < \beta$ such that $\{\alpha,
	\beta\} \in E$ and $h(\{\alpha, \beta\}) = k$. This completes the construction
	of $G$.

	Notice that, for all $\beta \in S$ and all natural numbers $k \geq 1$, if $\alpha^\beta_k$
	is defined, then $\alpha^\beta_k > C_\beta(\max[I_{k-1}])$. It follows that,
	if $N^{<}_G(\beta)$ is infinite, then it is a set of order type $\omega$
	converging to $\beta$. Therefore, $G$ is indeed a Hajnal-M\'{a}t\'{e} graph.
	The verification that the finite subgraphs of $G$ behave as desired is
	exactly as in the proof of Theorem A, so we omit it.

	It remains to show that $\chi(G) = \aleph_1$. Suppose for sake of contradiction
	that $c:S \rightarrow \bb{N}$ is a proper coloring of $G$. Let
	$S^* = \{\alpha \in S \mid c \restriction (S \cap \alpha) = f_\alpha\}$. By
	the properties of our $\diamondsuit$ sequence, $S^*$ is stationary and, moreover, $\langle C_\alpha \mid \alpha
	\in S^* \rangle$ is a club-guessing sequence. By Proposition
	\ref{club_guessing_decomposition_prop}, we can fix $k \in \bb{N}$ such that,
	letting $S^*_k := \{\alpha \in S^* \mid c(\alpha) = k\}$,
	$\langle C_\alpha \mid \alpha \in S^*_k \rangle$ is a club-guessing sequence.

	By Lemma \ref{type_lemma}, we can find $\alpha^* < \beta^*$ in $S^*_k$ such
	that, for all $j \leq k$, $C_{\alpha^*}[I_j]$ and $C_{\beta^*}[I_j]$ are
	disjoint and $\tp(C_{\alpha^*}[I_j], C_{\beta^*}[I_j]) = t^{n_j}_{s_j}$.
	It follows that, when considering $\beta^*$ in the construction of $G$,
	the answer to the question about $k$ was ``yes", since $\alpha^*$ is a witness.
	There is therefore an $\alpha \in S \cap \beta^*$ such that $f_{\beta^*}(\alpha)
	= k$ and $\{\alpha, \beta^*\} \in E$. But then we have
	\[
		c(\alpha) = f_{\beta^*}(\alpha) = k = c(\beta^*),
	\]
	contradicting the assumption that $c$ is a proper coloring of $G$. Thus,
	$\chi(G) = \aleph_1$, and we have completed the proof.
\end{proof}

\section{Questions}

We end with a couple of questions that remain open. First, as mentioned above,
it is unknown whether the existence of graphs of size $\aleph_1$ such that
$f_G$ grows arbitrarily quickly follows simply from $\mathsf{ZFC}$.

\begin{question}
	Is it true in $\mathsf{ZFC}$ that, for every function $f:\bb{N} \rightarrow
	\bb{N}$, there is a graph $G$ such that $|G| = \chi(G) = \aleph_1$ and, for all
	sufficiently large $k \in \bb{N}$, $f_G(k) \geq f(k)$?
\end{question}

Next, our method only produces graphs of chromatic number precisely $\aleph_1$.
It is unclear whether we can get such graphs of arbitrarily large chromatic number.

\begin{question}
	Is it true that, for every function $f:\bb{N} \rightarrow \bb{N}$ and every
	cardinal $\kappa$, there is a graph $G$ such that $\chi(G) \geq \kappa$ and,
	for all sufficiently large $k \in \bb{N}$, $f_G(k) \geq f(k)$?
\end{question}

\bibliographystyle{plain}
\bibliography{bib}

\begin{thebibliography}{10}

\bibitem{avart_luczak_rodl}
Christian Avart, Tomasz \L{}uczak, and Vojt\v{e}ch R\"{o}dl.
\newblock On generalized shift graphs.
\newblock {\em Fund. Math.}, 226(2):173--199, 2014.

\bibitem{de_bruijn_erdos}
N.~G. de~Bruijn and P.~Erd\H{o}s.
\newblock A colour problem for infinite graphs and a problem in the theory of
  relations.
\newblock {\em Nederl. Akad. Wetensch. Proc. Ser. A. {\bf 54} = Indagationes
  Math.}, 13:369--373, 1951.

\bibitem{erdos_circuits_and_subgraphs}
P.~Erd\H{o}s.
\newblock On circuits and subgraphs of chromatic graphs.
\newblock {\em Mathematika}, 9:170--175, 1962.

\bibitem{erdos_hajnal_chromatic_number}
P.~Erd\H{o}s and A.~Hajnal.
\newblock On chromatic number of graphs and set-systems.
\newblock {\em Acta Math. Acad. Sci. Hungar}, 17:61--99, 1966.

\bibitem{erdos_hajnal_szemeredi}
P.~Erd\H{o}s, A.~Hajnal, and E.~Szemer\'{e}di.
\newblock On almost bipartite large chromatic graphs.
\newblock In {\em Theory and practice of combinatorics}, volume~60 of {\em
  North-Holland Math. Stud.}, pages 117--123. North-Holland, Amsterdam, 1982.

\bibitem{erdos_unsolved_problems}
Paul Erd\H{o}s.
\newblock Some of my favourite unsolved problems.
\newblock In {\em A tribute to {P}aul {E}rd\H{o}s}, pages 467--478. Cambridge
  Univ. Press, Cambridge, 1990.

\bibitem{hajnal_mate}
Andr\'{a}s Hajnal and Attila M\'{a}t\'{e}.
\newblock Set mappings, partitions, and chromatic numbers.
\newblock In {\em Logic {C}olloquium '73 ({B}ristol, 1973)}, pages 347--379.
  North-Holland, Amsterdam, 1975.

\bibitem{jech}
Thomas Jech.
\newblock {\em Set theory}.
\newblock Springer Monographs in Mathematics. Springer-Verlag, Berlin, 2003.
\newblock The third millennium edition, revised and expanded.

\bibitem{komjath_note_on_hajnal_mate}
P\'{e}ter Komj\'{a}th.
\newblock A note on {H}ajnal-{M}\'{a}t\'{e} graphs.
\newblock {\em Studia Sci. Math. Hungar.}, 15(1-3):275--276, 1980.

\bibitem{komjath_shelah_forcing_constructions}
P\'{e}ter Komj\'{a}th and Saharon Shelah.
\newblock Forcing constructions for uncountably chromatic graphs.
\newblock {\em J. Symbolic Logic}, 53(3):696--707, 1988.

\bibitem{komjath_shelah_finite_subgraphs}
P\'{e}ter Komj\'{a}th and Saharon Shelah.
\newblock Finite subgraphs of uncountably chromatic graphs.
\newblock {\em J. Graph Theory}, 49(1):28--38, 2005.

\bibitem{cardinal_arithmetic}
Saharon Shelah.
\newblock {\em Cardinal arithmetic}, volume~29 of {\em Oxford Logic Guides}.
\newblock The Clarendon Press, Oxford University Press, New York, 1994.
\newblock Oxford Science Publications.

\end{thebibliography}

\end{document}